\documentclass{amsart}
 \usepackage{amssymb,amsmath,amsfonts,epsfig,latexsym}

 \newtheorem{theorem}{Theorem}

\newtheorem{proposition}[theorem]{Proposition}
\newtheorem{lemma}[theorem]{Lemma}

\theoremstyle{definition}
\newtheorem{remark}[theorem]{Remark}
\newtheorem{example}[theorem]{Example}

\def\Q{\ensuremath{\mathbb{Q}}}

\def\R{\ensuremath{\mathbb{R}}}

\def\cX{\ensuremath{\mathcal{X}}}
\def\cY{\ensuremath{\mathcal{Y}}}

\def\m{\ensuremath{\mathfrak{m}}}

\def\<{\ensuremath{\langle}}
\def\>{\ensuremath{\rangle}}

\DeclareMathOperator{\Spec}{Spec}

\DeclareMathOperator{\Trop}{Trop}
\DeclareMathOperator{\TTrop}{\mathfrak{Trop}}

\begin{document}

\title{Correction to ``Fibers of tropicalization"}

\author[Payne]{Sam Payne}
\address{Yale University, Department of Mathematics, 10 Hillhouse Ave, New Haven, CT 06511}
\email{sam.payne@yale.edu}

\begin{abstract}
This note explains an error in Proposition~5.1 of ``Fibers of tropicalization," Math. Z. 262 (2009), no. 2, 301--311, discovered by W.~Buczynska and F.~Sottile, and fills the resulting gap in the proof of the paper's main theorem.
\end{abstract}

\maketitle

Part (3) of Proposition~5.1 in \cite{tropicalfibers} claims that if $X$ is a subvariety of a torus $T$ containing the identity then there is a split surjection $\varphi: T \rightarrow T'$ such that the image of $X$ is a hypersurface and the intersection of the initial degeneration $X_0$ with the kernel of $\varphi_0$ is $\{1_T\}$.  This claim is false, and the following is a counterexample.  

\begin{example}
Suppose the characteristic of $K$ is not 2, and let $X$ be a curve in a three-dimensional torus containing all eight $2$-torsion points of $T$.  Then $X_0$ contains all eight $2$-torsion points of $T_0$ and, for any projection $\varphi$ from $T$ to a two dimensional torus, the kernel of $\varphi_0$ contains four $2$-torsion points, all of which are in $X_0$. 
\end{example}

\noindent The falsehood of part (3) of Proposition~5.1 leaves an essential gap in the proof of Theorem~4.1, which is the main result of \cite{tropicalfibers}.  This result has now been proved independently by different means, including nonarchimedean analysis \cite[Pro\-position~4.14]{Gubler12} and noetherian approximation \cite[Theorem~4.2.5]{tropicallifting}.  The original proposed method of proof using split surjections of tori to decrease the codimension may be of independent interest, but the error in Proposition~5.1 interferes with the reduction to the hypersurface case.  

Here, we complete the proof of Theorem~4.1 via the original method of split surjections of tori by projecting even further, onto a torus of dimension equal to $\dim X$, and using the going-down theorem for finite extensions of integral domains.

\begin{remark}
The second main result of \cite{tropicalfibers} is Corollary~4.2, which says that the fibers of the classical tropicalization map from $X(K)$ to $N_G$ are Zariski dense.  The error in Proposition~5.1 does not create a serious gap in the proof of this weaker result; the original arguments can be modified, as follows, to obtain Corollary~4.2 without deducing it from Theorem~4.1.  

Parts (1) and (2) of Proposition~5.1 reduce Corollary~4.2 to the hypersurface case, and Proposition~6.1 shows that if $X$ is a hypersurface then $\Trop^{-1}(v) \cap X(K)$ is infinite.  The argument in Section~6.3 then goes through with $\underline x$ and $\TTrop^{-1}(\underline x)$ replaced by $v$ and $\Trop^{-1}(v)$, respectively.
\end{remark}

To prove Theorem~4.1, we first consider the special case when $X$ is a torus. 

\begin{lemma}  \label{lem:torus}
For any $v \in N_G$ and $\underline x \in T_v(k)$, the fiber $\TTrop^{-1}(\underline x)$ is Zariski dense in $T$.
\end{lemma}

\begin{proof}
After translation, we may assume that $v = 0$ and $\underline x = 1_T$. If $T = K^*$ is one-dimensional then $\TTrop^{-1}(1_{T})$ is identified with the subset $1 + \m$ of the valuation ring $R$, which is infinite and hence Zariski dense.  For $T$ of arbitrary dimension $n$, choose an isomorphism from $T$ to  $(K^*)^n$.  Then $\Trop^{-1}(1_T)$ is identified with the product of dense sets $(1+\m)^n$, and is therefore dense.
\end{proof}

Next, we choose a suitable projection from $X$ to a torus of dimension equal to $\dim X$.  Say $X$ has dimension $d$.  Recall that the set of $v$ in $N_\R$ such that $X_v$ is nonempty is the underlying set of a finite polyhedral complex $\Delta$ of pure dimension $d$.  Projecting along a general rational subspace of codimension $d$ in $N_\Q$ maps each face of $\Delta$ isomorphically onto its image.  Such a projection corresponds to a split surjection of tori $\varphi: T \rightarrow T'$ with the property that, for each $v' \in N'_G$, the preimage
\[
\phi^{-1}(v') \cap \Trop(X)
\]
is finite, where $\phi: N_G \rightarrow N'_G$ is the linear map induced by $\varphi$.

Say $v_0, \ldots, v_s$ are the finitely many preimages in $\Delta$ of $\phi(v)$, where $v_0 = v$.  Each of the schemes $T_{v_0}, \ldots, T_{v_s}$ over $\Spec R$ contains the torus $T_K$ as an open subscheme, and the morphisms $\varphi_{v_i} : T_{v_i} \rightarrow T'_{\phi(v)}$ all agree on $T_K$.  Therefore, we can glue these schemes and morphisms to get
\[
T_{v_0} \cup \cdots \cup T_{v_s} \rightarrow T'_{\phi(v)}
\]
over $\Spec R$.  We write $\Phi$ for the restriction of this morphism to the closure of $X$ and prove Theorem~4.1 using the following technical result.

\begin{proposition} \label{prop:finite}
The morphism $\Phi:\cX_{v_0} \cup \cdots \cup \cX_{v_s} \rightarrow T'_{\phi(v)}$ is finite.
\end{proposition}

\noindent Here $\cX_{v_i}$ denotes the closure of $X$ in $T_{v_i}$.  The schemes $\cX_{v_i}$ and $T'_{\phi(v)}$ over $\Spec R$ are not noetherian, but this does not create any additional difficulties.  We deduce Theorem~4.1 from Proposition~\ref{prop:finite} using the going-down theorem for finite extensions of an integrally closed domain, which has no noetherian hypothesis.

\begin{proof}[Proof of Theorem~4.1]
First, we show that $R[M']^{\phi(v)}$ is an integrally closed domain.  Translating by a point in $\Trop^{-1}(\phi(v))$ induces an isomorphism from $R[M']^{\phi(v)}$ to $R[M']$ which is a localization of a polynomial ring over the integrally closed domain $R$, and hence an integrally closed domain, by \cite[Proposition~17.B(2)]{Matsumura80} and \cite[Example~9.3]{Matsumura89}.

By  Lemma~\ref{lem:torus}, the points $x'$ in $\TTrop^{-1}(\varphi_v(\underline x)$ are dense in $T'$.  By the going-down theorem for finite extensions of an integrally closed domain \cite[Theorem~9.4(ii)]{Matsumura89}, for each such $x'$ there is a point $x \in X(K)$ specializing to $\underline x$ such that $\varphi(x) = x'$.    This shows that the image of $\TTrop^{-1}(\underline x) \cap X(K)$ is Zariski dense in $T'$.  Since $\varphi$ is finite, it follows that $\TTrop^{-1}(\underline x) \cap X(K)$ is Zariski dense in $X$, as required.
\end{proof}

It remains to prove Proposition~\ref{prop:finite}.  To do this, we work with $G$-admissible fans and the associated toric schemes over $\Spec R$, as in \cite{Gubler12}, to which we refer the reader for details of these constructions.  Choose a $G$-admissible fan structure $\Sigma'$ on $N'_\R \times \R_{\geq 0}$ that contains $\R_{\geq 0} \cdot (\phi(v), 1)$ as a 1-dimensional face.  Let $\Sigma$ be a $G$-admissible fan on $N_\R \times \R_{\geq 0}$ such that
\begin{enumerate}
\item For each face $\tau \in \Delta$, the cone $\R_{\geq 0} \cdot (\tau \times 1)$ is a union of faces of $\Sigma$.
\item For each face $\sigma \in \Sigma$, the image $(\phi \times 1)(\sigma)$ is contained in a face of $\Sigma'$.
\end{enumerate}
Let $\cY_\Sigma$ and $\cY_{\Sigma'}$ be the corresponding toric schemes over $\Spec R$, as defined in \cite[Section~7]{Gubler12}, containing $T_K$ and $T'_K$ as dense open subschemes, respectively, and let $\cX_\Sigma$ be the closure of $X$ in $\cY_\Sigma$.  Since $\Sigma'$ contains $\R_{\geq 0} \cdot (\phi(v), 1)$, the toric scheme $\cY_{\Sigma'}$ contains $T'_{\phi(v)}$ as an open subscheme.  Then, by (1) and (2), the cones $\R_{\geq 0} (v_i, 1)$ must be cones in $\Sigma$, for $i = 0, \ldots, s$, and hence $\cX_\Sigma$ contains $\cX_{v_0} \cup \cdots \cup \cX_{v_s}$ as an open subscheme as well.

The morphism $\varphi: T_K \rightarrow T'_K$ extends to a morphism of toric schemes from $\cY_\Sigma$ to $\cY_{\Sigma'}$, by \cite[11.9]{Gubler12}, and we write $\overline \Phi$ for the restriction of this morphism to $\cX_\Sigma$.  The morphism $\Phi$ in Proposition~\ref{prop:finite} is the restriction of $\overline \Phi$ to $\cX_{v_0} \cup \cdots \cup \cX_{v_s}$.

\begin{lemma} \label{lem:proper}
The morphism $\overline \Phi: \cX_\Sigma \rightarrow \cY_{\Sigma'}$ is proper and of finite presentation.
\end{lemma}

\begin{proof}
By \cite[Proposition~11.12]{Gubler12} and (1), the scheme $\cX_\Sigma$ is proper over $\Spec R$.  Since $\cY_{\Sigma'}$ is separated over $\Spec R$ \cite[Lemma~7.8]{Gubler12}, it follows that $\cX_\Sigma$ is proper over $\cY_{\Sigma'}$ \cite[Corollary~5.4.3(i)]{EGA2}.

Since $\Sigma$ is $G$-admissible and $G$ is divisible, the toric scheme $\cY_\Sigma$ is of finite presentation over $\Spec R$ \cite[Proposition~6.7]{Gubler12}, and in particular it is of finite type over $\Spec R$.  Therefore, $\cX_\Sigma$ is of finite type over $\Spec R$.  Since $\cX_\Sigma$ is the closure of its generic fiber, by construction, it is also flat, and hence of finite presentation over $\Spec R$ \cite[Corollary~3.4.7]{RaynaudGruson71}.  Hence $\cX_\Sigma$ is also of finite presentation over $\cY_\Sigma$ \cite[Proposition~1.6.2(v)]{EGA4.1}.
\end{proof}

\begin{proof}[Proof of Proposition~\ref{prop:finite}]
By \cite[Lemma~11.6]{Gubler12}, the union $\cX_{v_0} \cup \cdots \cup \cX_{v_s}$ is the full preimage of $T'_{\phi(v)}$ under $\overline \Phi$.  Therefore, by Lemma~\ref{lem:proper}, $\Phi$ is proper and of finite presentation.  To show that $\Phi$ is finite, it remains to show that it has finite fibers \cite[Theorem~8.11.1]{EGA4.3}.

We begin with the general fiber of $T'_{\phi(v)}$.  If $x' \in T'(K)$, then the tropicalization of the fiber $\Phi^{-1}(x')$ must be contained in $\phi^{-1}(\Trop(x')) \cap \Trop(X)$, which is finite by construction.  Therefore, the fiber must be zero-dimensional and hence finite, since $\Phi$ is of finite presentation.

It remains to check the special fiber of $T'_{\phi(v)}$.  Suppose $\underline x'$ is in $T'_{\phi(v)}(k)$.  The preimage of $T'_{\phi_v}(k)$ is the disjoint union $X_{v_0} \cup \cdots \cup X_{v_s}$.  Therefore, the tropicalization of $\Phi^{-1}(\underline x') \cap X_{v_i}$ must be contained in the fiber over $0$ under the natural map from the star of $v_i$ in $\Trop(X)$ to $N'_\R$, which is, again, finite by construction.  It follows that the fiber over $\underline x'$ must be zero-dimensional, and hence finite, as required.
\end{proof}

\begin{remark}
In addition to the error in Proposition~5.1, there is an unrelated sign error in Section~3 of \cite{tropicalfibers}, which appears in four places.  At the bottom of p.~305, the tilted group ring should be defined as
\[
R[M]^v = \bigoplus_{u \in M} \m^{-\<u,v\>}.
\]
The quotient $k[T_v]$ of this ring by the ideal generated by $\m$ is then $\bigoplus_{u \in M} k^{-\<u,v\>}$, the weight function on monomials is given by $bx^u \mapsto \nu(b) + \<u,v\>$, and the weight of the monomial $a_{u,i} x^u t^i$ in Example~3.3 is $i + \<u,v\>$.
\end{remark}

\noindent \textbf{Acknowledgments.}  I am most grateful to W.~Bu\-czynska and F. Sottile for their careful reading and for bringing these mistakes to my attention.  This work was supported in part by NSF DMS 1068689 and completed during a visit to the Max Planck Institute for Mathematics in Bonn, Germany.

\bibliographystyle{amsalpha}
\bibliography{math}

\end{document}